\documentclass{amsart}

\usepackage{amssymb}
\newtheorem{theorem}{Theorem}[section]
\newtheorem{lemma}[theorem]{Lemma}
\newtheorem{proposition}[theorem]{Proposition}

\newtheorem{fact}[theorem]{Fact}

\theoremstyle{definition}

\theoremstyle{remark}
\newtheorem{remark}[theorem]{Remark}

\numberwithin{equation}{section}

\renewcommand{\dim}{\mathrm{dim}}

\newcommand{\dist}{\mathrm{dist}}
\newcommand{\conv}{\mathrm{conv}}

\newcommand{\R}{\mathbb{R}}
\newcommand{\N}{\mathbb{N}}
\newcommand{\C}{\mathbb{C}}
\newcommand{\K}{\mathbb{K}}

\newcommand{\X}{\mathrm{X}}

\newcommand{\Y}{\mathrm{Y}}

\newcommand{\B}{\mathbf{B}}
\newcommand{\I}{\mathbf{I}}

\renewcommand{\S}{\mathbf{S}}

\renewcommand{\mod}{/}

\renewcommand{\ae}{\stackrel{m}{\sim}}

\newcommand{\U}{\mathbf{U}}
\begin{document}

\title{Uniformly Convex-Transitive Function Spaces}

\author[F. Rambla]{Fernando Rambla-Barreno}
\address{Fernando Rambla\\ Universidad de C\'{a}diz, Departamento de Matem\'{a}ticas, 11510, Puerto Real, Spain}
\email{fernando.rambla@uca.es}

\author[J. Talponen]{Jarno Talponen}
\address{Jarno Talponen\\ University of Helsinki, Department of Mathematics and Statistics, Box 68, FI-00014 University of Helsinki,
Finland}
\email{talponen@cc.helsinki.fi}

\date{\today}

\begin{abstract}
We introduce a property of Banach spaces, called uniform convex-transitivity, which falls between almost transitivity
and convex-transitivity. We will provide examples of uniformly convex-transitive spaces. This property behaves nicely
in connection with some vector-valued function spaces. As a consequence, we obtain some new examples of convex-transitive Banach spaces.
\end{abstract}

\maketitle

\section{Introduction}

In this paper we study the symmetries of some well-known, in fact, almost classical Banach spaces.
We denote the closed unit ball of a Banach space $\X$ by $\B_{\X}$ and the unit sphere of $\X$ by
$\S_{\X}$. A Banach space $\X$ is called \emph{transitive} if for each $x\in \S_{\X}$ the orbit
$\mathcal{G}_{\X}(x)\stackrel{\cdot}{=}\{T(x)|\ T\colon \X\rightarrow \X\ \mathrm{is\ an\
isometric\ automorphism}\}=\S_{\X}$. If $\overline{\mathcal{G}_{\X}(x)}=\S_{\X}$ (resp.
$\overline{\conv}(\mathcal{G}_{\X}(x))=\B_{\X}$) for all $x\in\S_{\X}$, then $\X$ is called
\emph{almost transitive} (resp. \emph{convex-transitive}). These concepts are motivated by the
\emph{Banach-Mazur rotation problem} appearing in \cite[p.242]{Ba}, which remains unsolved. We
refer to \cite{BR2} and \cite{Ca0} for a survey and discussion on the matter.

The known concrete examples of convex-transitive spaces are scarce, and the ultimate aim of this
paper is to provide more examples by establishing the convex-transitivity of some vector-valued
function spaces and other natural spaces. It was first reported by Pelczynski and Rolewicz
\cite{PR} in 1962 that the space $L^p$ is almost transitive for $p\in [1,\infty)$ and
convex-transitive for $p=\infty$ (see also \cite{Rol}). Later, Wood \cite{Wo} characterized the
spaces $C_0^{\R}(L)$ whose norm is convex-transitive (see Preliminaries). Greim, Jamison and
Kaminska \cite{GJK} proved that if $\X$ is almost transitive and $1\leq p<\infty$, then the
Lebesgue-Bochner space $L^p(\X)$ is also almost transitive. Recently, an analogous study of the
spaces $C_0(L,\X)$ was done by Aizpuru and Rambla \cite{AR}, and some related spaces were studied
by Talponen \cite{conv}. For some other relevant contemporary results, see \cite{Ca?},
\cite{Kawamura} and \cite{Rambla}.

We will extend these investigations into the vector-valued convex-transitive setting, which
differs considerably in many respects from the scalar-valued almost transitive one. For this
purpose we will introduce a new concept which is (formally) stronger than convex-transitivity and
weaker than almost transitivity, called \emph{uniform convex-transitivity}. With the aid of this
class of Banach spaces we produce new natural examples of convex-transitive vector-valued function
spaces. The main results of this paper are the following:
\begin{itemize}
\item Characterization of locally compact Hausdorff spaces $L$ such that $C_0^{\R}(L)$ is uniformly convex-transitive.
\item If $\X$ is a uniformly convex-transitive Banach space, then so is $L_{\K}^{\infty}(\X)$.
\item If $\X$ and $C_0^{\R}(L)$ are uniformly convex-transitive, then so is $C_0^{\K}(L,\X)$.
\end{itemize}

\subsection*{Preliminaries}
The scalar field of a Banach space $\X$ is denoted by $\K$ and whenever there are several Banach
spaces under discussion, then $\K$ is the scalar field of the space denoted by $\X$. The open unit
ball of $\X$ is denoted by $\U_{\X}$. The group of rotations $\mathcal{G}_{\X}$ of $\X$ consists
of isometric automorphisms $T\colon\X\rightarrow\X$, the group operation being the composition of
the maps and the neutral element being the identity map $\I\colon \X\rightarrow\X$. We will always
consider $\mathcal{G}_{\X}$ equipped with the strong operator topology (SOT). An element $x\in
\S_{\X}$ is called a big point if $\overline{\conv}{\mathcal{G}(x)}=\B_{\X}$. Thus $\X$ is
convex-transitive if and only if each $x\in\S_{\X}$ is a big point.

Recall that a topological space is totally disconnected if each connected component of the space
is a singleton. In what follows $L$ is a locally compact Hausdorff space and $K$ is a compact
Hausdorff space, unless otherwise stated. In \cite{Wo} Wood characterized convex-transitive
$C_0^{\R}(L)$ spaces. Namely, $C_0^{\R}(L)$ is convex-transitive if and only if $L$ is totally
disconnected and for every regular probability measure $\mu$ on $L$ and $t\in L$ there exists a
net $\{\gamma_{\alpha}\}_{\alpha}$ of homeomorphisms on $L$ such that the net $\{\mu\circ
\gamma_{\alpha}\}_{\alpha}$ is $\omega^{\ast}$-convergent to the Dirac measure $\delta_{t}$. The
above mapping $\mu\circ \gamma_{\alpha}$ is given by $\mu\circ
\gamma_{\alpha}(A)=\mu(\gamma_{\alpha}(A))$ for Borel sets $A\subset L$.

We refer to \cite{Lac} for background information on measure algebras and isometries of $L^p$-spaces and to
\cite{HHZ} for a suitable source to Banach spaces in general. In what follows $\Sigma$ is the
completed $\sigma$-algebra of Lebesgue measurable sets on $[0,1]$ and we denote by $m\colon
\Sigma\rightarrow \R$ the Lebesgue measure. Define an equivalence relation $\ae$ on $\Sigma$ by
setting $A\ae B$ if $m((A\cup B)\setminus (A\cap B))=0$.

Recall that a rotation $R$ on the space $C_0^{\K}(L,\X)$ is said to be of the \emph{Banach-Stone
type}, if $R$ can be written as
\[R(f)(t)=\sigma(t)(f\circ\phi(t)),\quad f\in C_0^{\K}(L,\X),\]
where $\phi\colon L\rightarrow L$ is a homeomorphism and $\sigma\colon L\rightarrow
\mathcal{G}_{\X}$ is a continuous map. A Banach space $\Y$ is said to be contained \emph{almost
isometrically} in a Banach space $\X$ if for each $\varepsilon>0$ there is a linear map
$\psi\colon \Y\rightarrow \X$ such that
\[||y||_{\Y}\leq ||\psi(y)||_{\X}\leq (1+\varepsilon)||y||_{\Y}\quad \mathrm{for}\ y\in \Y.\]

\section{Uniform convex-transitivity}
Provided that the space $\X$ under discussion is understood, we denote
\[C_{n}(x)=\left\{\sum_{i=1}^{n}a_{i}T_{i}(x)|\ T_{1},\ldots,T_{n}\in \mathcal{G}_{\X},\ a_{1},\ldots,a_{n}\in [0,1],\ \sum_{i=1}^{n}a_{i}=1\right\}\]for $n\in \N$ and $x\in \S_{\X}$.
We call a Banach space $\X$ \emph{uniformly convex-transitive} if for each $\varepsilon>0$ there
exists $n\in \N$ satisfying the following condition: For all $x\in \S_{\X}$ and $y\in \B_{\X}$ it
holds that $\dist(y,C_{n}(x))\leq \varepsilon$, that is
\[\lim_{n\rightarrow \infty}\sup_{x \in \S_{\X}, y \in \B_{\X}}\dist(y,C_{n}(x))=0.\]
We denote by $K_{\varepsilon}$ the least integer $n$, which satisfies the above inequality involving $\varepsilon$
and such $K_{\varepsilon}$ is called \emph{the constant of uniform convex transitivity of $\X$ associated to $\varepsilon$.}
We call $x\in \S_{\X}$ a \emph{uniformly big point} if
\[\lim_{n\rightarrow \infty}\sup_{y\in \B_{\X}}\dist(y,C_{n}(x))=0.\]

Clearly almost transitive spaces are uniformly convex-transitive, and uniformly convex-transitive
spaces are convex-transitive. It is well-known that $C^{\C}(S^{1})$ is a convex-transitive,
non-almost transitive space, and it is easy to see (see e.g. the subsequent Theorem \ref{thm:
C0char}) that it is even uniformly convex-transitive. Unfortunately, we have not been able so far
to find an example of a convex-transitive space which is not uniformly convex-transitive. However,
we suspect that such examples exist and we note that the absence of such a complicated space would
make some proofs regarding convex-transitive spaces much more simple. Observe that the canonical
unit vectors $e_{k}\in \ell^{1}$ are far from being uniformly big points:
\[\lim_{n\rightarrow \infty}\sup_{y\in \B_{\ell^{1}}}\dist(y,C_{n}(e_{k}))=1,\]
even though they are big points, i.e. $\overline{\conv}(\mathcal{G}_{\ell^{1}}(e_{k}))=\B_{\ell^{1}}$ for $k\in \N$.
In any case, we will provide examples of uniformly convex-transitive spaces, most of which are not previously known
to be even convex-transitive.

We note that if $\X$ is convex-transitive and there exists a uniformly big point $x\in \S_{\X}$, then each $y\in \S_{\X}$
is a uniformly big point. This does not mean, a priori, that $\X$ should be uniformly convex-transitive.
Next we give an equivalent condition to uniform convex transitivity, which is more applicable in calculations
than the condition introduced above.
\begin{proposition}
Let $\X$ be a Banach space. The following condition of $\X$ is equivalent to $\X$ being uniformly
convex-transitive: For each $\varepsilon>0$ there is $N_{\varepsilon}\in \N$ such that for each
$x\in \S_{\X}$ and $y\in \B_{\X}$ there are $T_{1},\ldots, T_{N_{\varepsilon}}\in
\mathcal{G}_{\X}$ such that
\begin{equation}\label{eq: bullet}
\begin{array}{l}
\left|\left|y-\frac{1}{N_{\varepsilon}}\sum_{i=1}^{N_{\varepsilon}}T_{i}(x)\right|\right|\leq \varepsilon.
\end{array}
\end{equation}
\end{proposition}
\begin{proof}
It is clear that \eqref{eq: bullet} implies uniform convex transitivity, even for the value
$K_{\varepsilon}=N_{\varepsilon}$ for each $\varepsilon>0$. Towards the other direction, let $\X$
be a uniformly convex-transitive Banach space, $\varepsilon>0$ and $x\in \S_{\X},\ y\in \B_{\X}$.
Let $K$ be the constant of uniform convex-transitivity of $\X$ associated to
$\frac{\varepsilon}{4}$. Then there are $a_{1},\ldots ,a_{K}\in [0,1],\ \sum_{i}a_{i}=1$ and
$T_{1},\ldots,T_{K}\in \mathcal{G}_{\X}$ such that
\begin{equation*}
\begin{array}{l}
\left|\left|y-\sum_{i=1}^{K}a_{i}T_{i}(x)\right|\right|\leq \frac{\varepsilon}{4}.
\end{array}
\end{equation*}
Put $m=\lceil\frac{4K}{\varepsilon}\rceil\in \N$, so that $K\cdot\frac{1}{m}\leq
\frac{\varepsilon}{4}$. Next we define an $m$-uple $(S_{1},\ldots,S_{m})\subset \mathcal{G}_{\X}$
as follows: For each $j\in \{1,\ldots,m\}$ and $i\in \{1,\ldots,K\}$ we put $S_{j}=T_{i}$ if
$\lceil m\sum_{n<i}a_{n}\rceil < j\leq \lfloor m\sum_{n\leq i}a_{n}\rfloor$. (Here
$\sum_{\emptyset}a_{n}=0$.) By applying the triangle inequality several times, we obtain that
\begin{equation*}
\begin{array}{l}
\left|\left|y-\frac{1}{m}\sum_{j=1}^{m}S_{j}(x)\right|\right|\leq \varepsilon.
\end{array}
\end{equation*}
Hence it suffices to put $N_{\varepsilon}=m=\lceil\frac{4K}{\varepsilon}\rceil$, where $K$ depends
only on the Banach space $\X$ and the value of $\varepsilon$.
\end{proof}
In what follows, we will apply the constant $N_{\varepsilon}$ freely without explicit reference to the
previous proposition, and if there is no danger of confusion, also without mentioning explicitly
$\X$ and $\varepsilon$, either.

The following condition on the locally compact space $L$ turns out to be closely related to the
uniform convex-transitivity of $C_0^{\K}(L)$:
\begin{enumerate}
\item[$(\ast)$]{For each $\varepsilon>0$ there is $M_{\varepsilon}\in \N$ such that for every non-empty open subset $U\subset L$
and compact $K\subset L$ there are homeomorphisms $\phi_{1},\ldots,\phi_{M_{\varepsilon}}\colon
L\rightarrow L$ with
\[\frac{1}{M_{\varepsilon}}\sum_{i=1}^{M_{\varepsilon}}\chi_{\phi_{i}^{-1}(U)}(t)\geq 1-\varepsilon\quad \mathrm{for}\ t\in K.\]}
\end{enumerate}
This condition should be compared with the conditions found by Cabello (see \cite[p.110-113]{Ca?},
especially condition (g)), which characterize the convex transitivity of $C_0(L)$. Next we will
give this characterization the uniformly convex-transitive counterpart. If $L$ is a locally compact Hausdorff space,
by $\alpha L$ we denote its one-point compactification and if
$L$ is noncompact, we denote such point by $\infty$. Prior to the theorem we need the following
two lemmas.


\begin{lemma}\label{lm: from Ra}(\cite[Thm. 3.1]{Rambla})
Let $T$ be a normal topological space with $\dim\ T \leq 1$. If $F\subseteq T$ is a closed subset and $f: F \to S_{\C}$ is a continuous
map, then $f$ admits a continuous extension $g: T \to S_{\C}$.
\end{lemma}

\begin{lemma}\label{lm: technical uct}
Let $L$ be a locally compact, Hausdorff, $0$-dimensional space. Then for every $g \in
\B_{C_0^{\R}(L)}$ and $k\in \N$ there exist disjoint clopen sets $C_1, C_2, \dots, C_{2k-1}$ such
that the function $h \in \B_{C_0^{\R}(L)}$ defined by $h=\sum_{i=1}^{2k-1}\frac{i-k}{k}\chi_{C_i}$
satisfies $\|h-g\| \leq \frac{3}{2k}$.
\end{lemma}

\begin{proof}
We regard $g$ as defined in $\alpha L$. Consider $i \in \{-k,-k+1,\dots,k-1\}$ and let
$K_i=g^{-1}[\frac{i}{k},\frac{i+1}{k}]$. Every $x \in K_i$ has a clopen neighbourhood $A_x$ such
that $g(A_x)\subseteq [\frac{2i-1}{2k},\frac{2i+3}{2k}]$. By compactness there exist $x_1, \dots,
x_n$ such that $K_i \subseteq \bigcup_{j=1}^n A_{x_j}\stackrel{\cdot}{=}B_i$. Finally, define
$C_0=B_0$, $C_1=B_{-k} \setminus C_0$, \dots, $C_k=B_{-1}\setminus (C_0 \cup \dots \cup C_{k-1})$,
$C_{k+1}=B_1\setminus (C_0 \cup \dots \cup C_k)$, \dots, $C_{2k-1}=B_{k-1}\setminus (C_0 \cup
\dots \cup C_{2k-2})$. Note that the $C_i$'s are a partition of $\alpha L$.

Now take $h: L \to \R$ given by $h=\sum_{i=1}^{2k-1}\frac{i-k}{k}\chi_{C_i}$. It is easy to check
that $\|h-g\| \leq \frac{3}{2k}$ and $h \in \B_{C_0^{\R}(L)}$.
\end{proof}

\begin{theorem}\label{thm: C0char}
Let $L$ be a locally compact Hausdorff space. The space $C_0^{\R}(L)$ is uniformly
convex-transitive if and only if $L$ is totally disconnected and satisfies $(\ast)$. If the space
$C_0^{\C}(L)$ is uniformly convex-transitive, then $L$ satisfies $(\ast)$. Moreover, if $\dim(\alpha L)\leq 1$,
then also the converse implication holds.
\end{theorem}
Before the proof we comment on the above assumptions.
\begin{remark}\label{remark}
The spaces $C^{\R}(S^{1},\R^{2})$ and $C^{\C}(S^{1},\C)$ are uniformly convex-transitive, their rotations are of the
Banach-Stone type, and clearly $S^{1}$, $\mathcal{G}_{\R^{2}}$ and $\mathcal{G}_{\C}$ are not totally disconnected.
\end{remark}

\begin{proof}[Proof of Theorem \ref{thm: C0char}]
Let us first consider the {\it only if} directions. Since uniformly convex-transitive spaces are
convex-transitive, we may apply Wood's characterization for convex-transitive $C_0^{\R}(L)$ spaces,
and thus we obtain that $L$ must be totally disconnected. Let $C_0^{\K}(L),\ \K\in\{\R,\C\},$ be
uniformly convex-transitive. Next we aim to check that $L$ satisfies $(\ast)$, so let $U\subset L$
be a non-empty open subset and $K\subset L$ a compact subset. Fix $x_{0}\in U$. Since $\alpha L$
is normal, there exist continuous functions $f,g: \alpha L \to [-1,1]$ satisfying $f(\alpha L
\setminus U)=\{0\}$, $f(x_0)=1$, $g(K)=\{1\}$ and $g(\infty)=0$. Since both functions vanish at
infinity, we can consider that $f,g \in \S_{C_0^{\K}(L)}$.

Fix $\varepsilon>0$ appearing in condition $(\ast)$. Let $N_{\varepsilon}$ be the associated
constant provided by the uniform convex-transitivity and condition \eqref{eq: bullet}. Then by the
definition of $N_{\varepsilon}$ and the Banach-Stone characterization of rotations of
$C_0^{\K}(L)$ we obtain that there are continuous functions $\sigma_{1},\ldots
,\sigma_{N_{\varepsilon}}\colon L\rightarrow \K$ and homeomorphisms
$\phi_{1},\ldots,\phi_{N_{\varepsilon}}\colon L\rightarrow L$ such that
\begin{equation}\label{eq: gNe}
\begin{array}{l}
\left|\left|g-\frac{1}{N_{\varepsilon}}\sum_{i=1}^{N_{\varepsilon}}\sigma_{i}(f \circ\phi_{i})\right|\right|\leq\varepsilon.
\end{array}
\end{equation}
In particular, this yields for each $t\in K$ that
\begin{equation*}
\begin{array}{lcl}
\varepsilon &\geq& |1-\frac{1}{N_{\varepsilon}}\sum_{i=1}^{N_{\varepsilon}}\sigma_{i}f(\phi_{i}(t))|=|\frac{1}{N_{\varepsilon}}\sum_{i=1}^{N_{\varepsilon}}1-\sigma_{i}f(\phi_{i}(t))|\\
         &\geq& \frac{1}{N_{\varepsilon}}\sum_{i=1}^{N_{\varepsilon}}\chi_{L\setminus \phi_{i}^{-1}(U)}(t),
\end{array}
\end{equation*}
where we applied the fact that $f$ vanishes outside $U$. This justifies $(\ast)$ for
$M_{\varepsilon}=N_{\varepsilon}$.

Let us see the {\it if} direction for $C_0^{\R}(L)$. Let $k \in \N, f \in \S_{C_0^{\R}(L)}$ and $g
\in \B_{C_0^{\R}(L)}$. We may assume $\max f=1$. Take $h$ as in Lemma \ref{lm: technical uct},
i.e. $h=\sum_{i=1}^{2k-1}\frac{i-k}{k}\chi_{C_i}$ with each $C_i$ clopen and $\|h-g\| \leq
\frac{3}{2k}$.

Note that $K\stackrel{\cdot}{=}\bigcup_{i=1}^{2k-1}C_i$ is compact and apply $(*)$ to this $K$,
the subset $U\stackrel{\cdot}{=}\{t\in L: f(t) >1-k^{-1}\}$ and $\varepsilon=\frac{1}{k}$. Write
$M\stackrel{\cdot}{=}M_\varepsilon$. There exist homeomorphisms $\phi_1, \dots, \phi_M$ such that
if $t \in K$ then $\frac{1}{M}\sum_{i=1}^M\chi_{\phi_i^{-1}(U)}(t) \geq 1-k^{-1}$. For each $j \in
\{1, \dots, 2k-1\}$, define $B_j=\bigcup_{s=j}^{2k-1}C_s$ and let $T_j$ be the rotation on
$C_0^{\R}(L)$ given by $T_jx=(\chi_{B_j}-\chi_{B_j^c})\cdot x$ if $j\leq k$ and
$T_jx=(\chi_{B_j}-\chi_{B_j^c}+2\chi_{L\setminus K})\cdot x$ if $j > k$. Now only a few
calculations are needed to see that
\begin{equation*}
\begin{array}{l}
\left|\left|g-\frac{1}{M(2k-1)}\sum_{j=1}^{2k-1}\sum_{i=1}^MT_j(f\circ \phi_i)\right|\right| \leq 6k^{-1}
\end{array}
\end{equation*}
and thus $C_0^{\R}(L)$ is uniformly convex transitive.

In order to justify the last claim it is required to verify that if $L$ satisfies $\dim(\alpha
L)\leq 1$ and $(\ast)$, then $C_0^{\C}(L)$ is uniformly convex-transitive. Let $k\in \N$ and let
$M_{k}$ be the corresponding constant in condition $(\ast)$ associated to value $k^{-1}$. Fix
$f\in \S_{C_0^{\C}(L)}$ and $g\in\B_{C_0^{\C}(L)}$. We may assume without loss of generality,
possibly by multiplying $f$ with a suitable complex number of modulus $1$, that $f(t_{0})=1$ for a
suitable $t_{0}\in L$. Let $U\stackrel{\cdot}{=}\{t\in L:\ |1-f(t)|<k^{-1}\}$ and $K=\{t\in L:\
|g(t)|\geq k^{-1}\}$. Let $\phi_{1},\ldots,\phi_{M_{k}}\colon L\rightarrow L$ be homeomorphisms
such that $\frac{1}{M_{k}}\sum_{i=1}^{M_{k}}\chi_{\phi_{i}^{-1}(U)}(t)\geq 1-k^{-1}$ for $t\in K$.
This means that the average
\begin{equation}\label{eq: F}
\begin{array}{l}
F\stackrel{\cdot}{=}\frac{1}{M_{k}}\sum_{i=1}^{M_{k}}f\circ \phi_{i}\in \B_{C_0^{\C}(L)}
\end{array}
\end{equation}
satisfies $|1-F(t)|\leq 3k^{-1}$ for each $t\in K$.

Next we will define some auxiliary mappings. Put $\alpha\colon \S_{\C}\times [0,1]\rightarrow
\S_{\C};\ \alpha(z,s)=-i^{2s}z$. Note that this is a continuous map, and $\alpha(z,0)=-z$,
$\alpha(z,1)=z$ for $z\in\S_{\C}$. Taking into account Lemma \ref{lm: from Ra} with $T=\alpha L$,
let $\beta_{g}\colon L\rightarrow \S_{\C}$ be a continuous extension of the function $\frac{g(\cdot)}{|g(\cdot)|}$ defined on $K$.

For $j\in \{1,\ldots,k\}$ we define rotations on $C_0^{\C}(L)$ by putting
$e_{ja}(x)(t)=\beta_{g}(t)\cdot x(t)$ and $e_{jb}(x)(t)=\alpha(\beta_{g}(t),\min(1,\max(0,
k|g(t)|-j)))\cdot x(t)$. The main point above is that $(e_{ja}+e_{jb})(F)(t)=0$ for $(j,t)\in
\{1,\ldots,k\}\times L$ such that $|g(t)|\leq \frac{j}{k}$ and
$(e_{ja}+e_{jb})(F)(t)=2F(t)\beta_{g}(t)$ for $(j,t)\in \{1,\ldots,k\}\times L$ such that
$g(t)\geq \frac{j+1}{k}$. Thus, by using \eqref{eq: F} we obtain that
\begin{equation*}
\begin{array}{l}
\left|\left|\ |g(t)|\beta_{g}(t)-\frac{1}{2k}\sum_{j=1}^{k}(e_{ja}+e_{jb})(F)(t)\right|\right|\leq 2k^{-1}\quad \mathrm{for}\ t\in L.
\end{array}
\end{equation*}
Here $||g(\cdot)-|g(\cdot)|\beta_{g}(\cdot)||\leq k^{-1}$, so that $C_0^{\C}(L)$ is uniformly
convex-transitive.
\end{proof}

Note that Theorem \ref{thm: C0char} yields the fact that if $C_0^{\R}(L)$ is uniformly
convex-transitive, then so is $C_0^{\C}(L)$. By the above reasoning one can also see that if
$C_0^{\K}(L)$ is convex-transitive and $|L|>1$, then $L$ contains no isolated points and thus it
follows that each non-empty open subset of $L$ is uncountable. Cabello pointed out
\cite[Cor.1]{Ca?} that locally compact spaces $L$ having a basis of clopen sets $C$ such that
$L\setminus C$ is homeomorphic to $C$, have the property that $C_0^{\R}(L)$ is convex-transitive.
Consequently, this provides a route to the fact that the spaces $L^{\infty},\ \ell^{\infty}\mod
c_0$ and $C(\Delta)$ over $\R$, where $\Delta$ is the Cantor set, are convex-transitive. By
applying Theorem \ref{thm: C0char} and following Cabello's argument with slight modifications, one
arrives at the conclusion that these spaces are in fact uniformly convex-transitive. When studying
\cite{Ca?} it is helpful to observe that each occurence of 'basically disconnected' in the paper
must be read as \emph{totally disconnected}, (\cite{CaTa}).

It is quite easy to verify that if $L_{1},\ldots,L_{n}$, where $n\in \N$, are totally disconnected
locally compact Hausdorff spaces satisfying $(\ast)$, then so is the product $L_{1}\times \dots
\times L_{n}$. It follows that the space $C_0^{\R}(L_{1}\times \dots \times L_{n})$ (also known as
the injective tensor product $C_0^{\R}(L_{1})\hat{\otimes}_{\varepsilon} \dots
\hat{\otimes}_{\varepsilon} C_0^{\R}(L_{n})$, up to isometry) is uniformly convex-transitive.

\section{Uniform convex-transitivity of Banach-valued function spaces}

With a proof similar to that of lemma \ref{lm: technical uct}, we obtain the following:

\begin{lemma}\label{lm: technical uctx}
Let $L$ be a locally compact, Hausdorff, $0$-dimensional space and $X$ a Banach space over $\K$.
Given $g \in \B_{C_0^{\K}(L,X)}$ and $j\in \N$, there exist nonzero $x_1, \ldots, x_n \in \B_{\X}$
and disjoint clopen sets $C_1, C_2, \ldots, C_n\subset L$ such that the function $h \in
\B_{C_0^{\K}(L,\X)}$ defined by $h(t)=\sum_{i=1}^n \chi_{C_i}(t)x_i$ satisfies
$\|h-g\|<\frac{1}{j}$.
\end{lemma}

\begin{theorem}\label{thm: CLX}
Let $L$ be a locally compact Hausdorff space and $\X$ a Banach space over $\K$. Consider the
following conditions:
\begin{enumerate}
\item[(1)]{$L$ is totally disconnected and satisfies $(\ast)$, i.e. $C_0^{\R}(L)$ is uniformly convex-transitive.}
\item[(2)]{$\X$ is uniformly convex-transitive.}
\item[(3)]{$C_0^{\K}(L,\X)$ is uniformly convex-transitive.}
\end{enumerate}
We have the implication $(1)+(2)\implies (3)$. If the rotations of $C_0^{\K}(L,\X)$ are of the
Banach-Stone type and $\dim_{\K}(\X)\geq 1$, then $(3)\implies (\ast)+(2)$. If additionally
$\K=\R$ and $\mathcal{G}_{\X}$ is totally disconnected, then $(3)\implies (1)+(2)$.
\end{theorem}

Recall Remark \ref{remark} related to the last claim above.

\begin{proof}[Proof of Theorem \ref{thm: CLX}]
We begin by proving the implication $(1)+(2)\implies (3)$. Fix $k\in \N$. Then condition $(\ast)$
provides us with an integer $N_{k}$ associated to $\frac{1}{4k}$. Let $f\in \S_{C_0^{\K}(L,\X)}$
and $g\in \B_{C_0^{\K}(L,\X)}$. Take $h$ and $C_{1},\ldots,C_{n}\subset L$ as in Lemma $\ref{lm: technical uctx}$
with $j=2k$.

Let $B=\bigcup_{i=1}^n C_i$ and $K=\{t \in L: \|g(t)\| \geq k^{-1}\}$. Note that $B$ is a compact clopen set and $K\subset B$.
There are $y\in \S_{\X}$ and $T_{1},\ldots,T_{N_{k}} \in \mathcal{G}_{C_{0}^{\K}(L,\X)}$ such that
\begin{equation}\label{eq: yNk}
\begin{array}{l}
\left|\left|y-\left(\frac{1}{N_{k}}\sum_{i=1}^{N_{k}}T_{i}f\right)(t)\right|\right|<\frac{1}{k},\quad \mathrm{for}\ t\in B.
\end{array}
\end{equation}
Indeed, observe that the continuous map $L\rightarrow \R;\ t\mapsto ||f(t)||$ attains its supremum, the value $1$.
Thus, let $t_{0}\in L$ be such that $||f(t_{0})||=1$ and let $y=f(t_{0})\in \S_{\X}$.
Write $V=\{t\in L:\ ||f(t)-y||<\frac{1}{2k}\}$. By using $(\ast)$ there are homeomorphisms
$\sigma_{1},\ldots, \sigma_{N_{k}}\colon L\rightarrow L$ such that
\begin{equation}\label{eq: NkB}
\begin{array}{l}
\frac{1}{N_{k}}\sum_{i=1}^{N_{k}}\chi_{V}(\sigma_{i}(t))\geq 1-\frac{1}{4k},\quad \mathrm{for}\ t\in B.
\end{array}
\end{equation}
Let $T_{i}\in\mathcal{G}_{C_{0}^{\K}(L,\X)}$ be given by $(T_{i}F)(t)=F(\sigma_{i}(t))$ for $1\leq i\leq N_{k}$ and
$F\in C_{0}^{\K}(L,\X)$. Thus, for all $t\in B$ we obtain by \eqref{eq: NkB} and the definition of $V$ that
\begin{equation*}
\begin{array}{lll}
& &\left|\left|y-\left(\frac{1}{N_{k}}\sum_{i=1}^{N_{k}}T_{i}f\right)(t)\right|\right|\phantom{\bigg |}\\
&=& \left|\left|\frac{1}{N_{k}}\sum_{i=1}^{N_{k}}y-\frac{1}{N_{k}}\sum_{i=1}^{N_{k}}f(\sigma_{i}(t))\right|\right|\leq \frac{1}{N_{k}}\sum_{i=1}^{N_{k}}||y-f(\sigma_{i}(t))||\phantom{\bigg |}\\
&<&(1-\frac{1}{4k})\cdot \frac{1}{2k}+\frac{1}{4k}\cdot 2<\frac{1}{k}.\phantom{\bigg |}
\end{array}
\end{equation*}

Since $X$ is uniformly convex-transitive, there is an integer $2M=N_{\varepsilon}$ satisfying \eqref{eq: bullet} for the value
$\varepsilon=k^{-1}$. Let $S^{(i)}_{1},\ldots , S^{(i)}_{2M}\in \mathcal{G}_{\X}$ for $1\leq i\leq n$ such that
\begin{equation}\label{eq: xMl}
\begin{array}{l}
\left|\left|x_{i}-\frac{1}{2M}\sum_{l=1}^{2M}S^{(i)}_{l}(y)\right|\right|<k^{-1}\quad \mathrm{for}\ 1\leq i\leq n.
\end{array}
\end{equation}
Then for each $1\leq l\leq 2M$ we define a rotation on $C_0^{\K}(L,\X)$ by
\[R_{l}(F)(t)=\chi_{L\setminus B}(t)(-1)^{l}F(t)+\sum_{i=1}^{n}\chi_{C_{i}}S^{(i)}_{l}(F(t)),
\quad F\in C_0^{\K}(L,\X),\ t\in L.\] Indeed, this defines rotations, since the sets $L\setminus
B$ and $C_{i}$ are clopen. It is easy to see by combining \eqref{eq: yNk} and \eqref{eq: xMl} that
\begin{equation*}
\begin{array}{l}
\left|\left|g-\frac{1}{2M}\sum_{l=1}^{2M}R_{l}\frac{1}{N_{k}}\sum_{i=1}^{N_{k}}T_{i}f\right|\right|<2k^{-1}.
\end{array}
\end{equation*}
This verifies the first implication.

Next we will prove the implication $(3)\implies (\ast)+(2)$ under the assumption that the
rotations are of the Banach-Stone type. In fact, the verification of claim  $(3)\implies (\ast)$
reduces to the analogous scalar-valued case, which was treated in the proof of Theorem \ref{thm:
C0char}. Moreover, by using the Banach-Stone representation of rotations and functions of type
$f\otimes x, g\otimes y\in \S_{C^{\K}_{0}(L,\X)}$ it is easy to verify that the uniform
convex-transitivity of $C_0^{\K}(L,\X)$ implies that of $\X$.

Finally, let us prove the total disconnectedness of $L$ in the case when $\mathcal{G}_{\X}$ is
totally disconnected and $\K=\R$. Assume to the contrary that $L$ contains a connected subset $C$,
which is not a singleton. Pick $t,s\in C,\ t\neq s,$ and $x\in \S_{\X}$. Let $x^{\ast}\in
\S_{\X^{\ast}}$ with $x^{\ast}(x)=1$. Let $f,g\in \S_{C_0^{\R}(L)}$ be functions with disjoint
supports and such that $f(t)=g(s)=1$. Consider $f\otimes x, f\otimes x -g\otimes x\in
\S_{C_0^{\R}(L,\X)}$. Since $C_0^{\R}(L,\X)$ is convex-transitive we obtain that $f\otimes x
-g\otimes x\in\mathcal{G}_{C_0^{\R}(L,\X)}(f\otimes x)$.

It follows easily by using the Banach-Stone representation of rotations of $C_0^{\R}(L,\X)$ that
there exists a continuous map $\sigma\colon L\rightarrow \mathcal{G}_{\X}$ such that
\[x^{\ast}(\sigma(t)(x)),x^{\ast}(-\sigma(s)(x))>0.\]

By using the facts that $\sigma(t)\neq\sigma(s)$ and that $\mathcal{G}_{\X}$ is totally disconnected we obtain
that $\sigma(C)$ is not connected. However, we have a contradiction, since $\sigma(C)$ is a continuous image of
a connected set. This contradiction shows that $L$ must be totally disconnected.
\end{proof}

By following the argument in the previous proof with slight modifications one obtains an analogous
result in the convex-transitive setting.
\begin{theorem}
If $C_0^{\R}(L)$ is convex-transitive and $\X$ is a convex-transitive space over $\K$, then
$C_0^{\K}(L,\X)$ is convex-transitive.
\end{theorem}
\begin{proof}
The proof of Theorem \ref{thm: CLX} has the convex-transitive counterpart with convex combinations of rotations in place of averages
of rotations. Indeed, in the equation \eqref{eq: yNk} one uses the convex-transitivity of $C_0^{\R}(L)$ and the corresponding Banach-Stone
type rotations applied on $C_0^{\K}(L,\X)$. After equation \eqref{eq: yNk} the argument proceeds similarly. Note that in the
convex-transitive setting there does not exist, a priori, an upper bound $M$ depending only on $\epsilon$.
\end{proof}

Recall that the Lebesgue-Bochner space $L^p(\X)$ consists of strongly measurable maps $f\colon
[0,1]\rightarrow \X$ endowed with the norm
\[||f||_{L^p(\X)}^{p}=\int_{0}^{1}||f(t)||_{\X}^{p}\ \mathrm{d}t,\quad \mathrm{for}\ p\in [1,\infty)\]
and $||f||_{L^{\infty}(\X)}=\underset{t\in [0,1]}{\mathrm{ess\ sup}}||f(t)||_{X}$.
We refer to \cite{DU} for precise definitions and background information regarding the Banach-valued function spaces
appearing here.

Recall that $L^{\infty}$ is convex-transitive (see \cite{PR} and \cite{Rol}). Greim, Jamison and
Kaminska proved that $L^p(\X)$ is almost transitive if $\X$ is almost transitive and $1\leq
p<\infty$, see \cite[Thm. 2.1]{GJK}. We will present the analogous result for uniformly
convex-transitive spaces, that is, if $\X$ is uniformly convex-transitive, then $L^p(\X)$ are also
uniformly convex-transitive for $1\leq p\leq \infty$.

\begin{theorem}\label{thm: LPX}
Let $\X$ be a uniformly convex-transitive space over $\K$.
Then the Bochner space $L_{\K}^{p}(\X)$ is uniformly convex-transitive for $1\leq p\leq\infty$.
\end{theorem}
We will make some preparations before giving the proof. Suppose that $(A_{n})_{n\in\N}$ is a countable measurable
partition of the unit interval and $(x_{n})_{n\in\N}\subset\X$. We will use the short-hand notation
$F=\sum_{n}\chi_{A_{n}}x_{n}$ for the function $F\in L^{\infty}(\X)$ defined by $F(t)=x_{n}$ for a.e. $t\in A_{n}$ for each $n\in\N$.
The following two auxiliary observations are obtained immediately from the fact that the countably valued functions are dense
in $L^{\infty}(\X)$ and the triangle inequality, respectively.
\begin{fact}\label{fact1}
Consider $F=\sum_{n}\chi_{A_{n}}x_{n}$, where $(A_{n})$ is a measurable partition of $[0,1]$
and $(x_{n})\subset \B_{\X}$. Functions $F$ of such type are dense in $\B_{L^{\infty}(\X)}$.
\end{fact}

\begin{fact}\label{fact2}
Let $\X$ be a Banach space and $T_{1},...,T_{n}\in \mathcal{G}_{\X},\ n\in \N$. Assume that
$x,y,z\in \X$ satisfy $||y-\frac{1}{n}\sum_{i}T_{i}(x)||=\varepsilon\geq 0$ and
$||x-z||=\delta\geq 0$. Then $||y-\frac{1}{n}\sum_{i}T_{i}(z)||\leq \varepsilon +\delta$.
\end{fact}

\begin{proof}[Proof of Theorem \ref{thm: LPX}]

We mainly concentrate on the case $p=\infty$. Fix $k\in \N$, $x\in \S_{\X}$,
$(x_{n}),(y_{n})\subset\B_{\X}$ and measurable partitions $(A_{n})$ and $(B_{n})$ of the unit
interval. Let $N_{k}$ be the integer provided by the uniform convex transitivity of $\X$
associated to the value $\varepsilon=\frac{1}{k}$. Write
\[F=\sum_{n}\chi_{A_{n}}x_{n}\ \mathrm{and}\ G=\sum_{n}\chi_{B_{n}}y_{n}.\]
We assume additionally that $||F||=1$.

For each $n\in \N$ there are isometries $\{T_{i}^{(n)}\}_{i\leq N_{k}}\subset\mathcal{G}_{\X}$ such that
\begin{equation}\label{eq: yj}
\begin{array}{cc}
&\left|\left|\frac{1}{N_{k}}\sum_{i=1}^{N_{k}}T_{i}^{(n)}(x)-y_{n}\right|\right|<\frac{1}{k}\quad \mathrm{for}\ n\in\N.
\end{array}
\end{equation}

Observe that one obtains rotations on $L^{\infty}(\X)$ by putting
\[R_{i}(f)(t)=\sum_{n}\chi_{B_{n}}T_{i}^{(n)}(f(t))\]
for a.e. $t\in [0,1]$, where $f\in L^{\infty}(\X)$, $i\leq N_{k}$, and the above summation is understood in the sense of
pointwise convergence almost everywhere. We define a convex combination of elements of $\mathcal{G}_{L^{\infty}(\X)}$ by
\begin{equation*}
\begin{array}{l}
\mathrm{A}_{1}(f)=\frac{1}{N_{k}}\sum_{i=1}^{N_{k}}R_{i}(f),\quad f\in L^{\infty}(\X).
\end{array}
\end{equation*}
Condition \eqref{eq: yj} implies that
\begin{equation}\label{eq: GCC}
||G-\mathrm{A}_{1}(\chi_{[0,1]}x)||<\frac{1}{k}.
\end{equation}

By the definition of $F$ one can find $n_{0}\in\N$ such that $m(A_{n_{0}})>0$ and
\begin{equation}\label{eq: supAFx}
||x_{n_{0}}||_{\X}>1-\frac{1}{k}.
\end{equation}
Put $\Delta_{n}=[1-2^{-n},1-2^{-(n+1)}]$ for $n\leq k$.
By composing suitable bijective transformations one can construct measurable mappings
$g_{n}\colon [0,1]\rightarrow [0,1]$ and $\hat{g}_{n}\colon [0,1]\rightarrow [0,1]$ such that
\begin{equation}
g_{n}(A_{n_{0}})\ae [0,1]\setminus \Delta_{n}\ \mathrm{and}\ g_{n}([0,1]\setminus A_{n_{0}})\ae \Delta_{n},
\end{equation}
\begin{equation}\label{eq: mequiv}
\mathrm{the\ measure}\ \mu_{n}(\cdot)\stackrel{\cdot}{=}m(g_{n}(\cdot))\colon\Sigma\rightarrow\R\ \mathrm{is\ equivalent\ to}\ m
\end{equation}
and
\begin{equation}\label{eq: hatg}
\hat{g}_{n}\circ g_{n}(t)=t\quad \mathrm{for}\ \mathrm{a.e.}\ t\in [0,1]
\end{equation}
for each $n\leq k$.

Next we will apply some observations which appear e.g. in \cite{Greim_Lp} and \cite{Greim_Linfty}.
Denote by $\Sigma\setminus_{m}$ the quotient $\sigma$-algebra of Lebesgue measurable sets on $[0,1]$ formed by identifying the
$m$-null sets with $\emptyset$. Note that \eqref{eq: mequiv} gives in particular that the map
$\phi_{n}\colon\Sigma\setminus_{m}\rightarrow\Sigma\setminus_{m}$ determined by $\phi_{n}(A)\ae g_{n}(A)$ for $A\in \Sigma$
is a Boolean isomorphism for each $n\leq k$. Observe that $\hat{g}_{n}(A)\ae \phi_{n}^{-1}(A)$ for $A\in\Sigma$ and $n\leq k$.

By \eqref{eq: supAFx} there are rotations $\{T_{i}\}_{i\leq N_{k}}\subset \mathcal{G}_{\X}$ such that
\begin{equation}\label{eq: xsumd}
\begin{array}{cc}
&\left|\left|x-\frac{1}{N_{k}}\sum_{i=1}^{N_{k}}T_{i}(x_{n_{0}})\right|\right|_{\X}<\frac{2}{k}.
\end{array}
\end{equation}

According to \eqref{eq: hatg} we may define mappings $S_{i}\colon L^{\infty}(\X)\rightarrow L^{\infty}(\X)$ for
$n\leq k$ and $i\leq N_{k}$ by putting
\[S_{i}^{(n)}(F)(t)=T_{i}(F(\hat{g}_{n}(t)))\quad \mathrm{for\ a.e.}\ t\in[0,1],\ F\in L^{\infty}(\X).\]
By \eqref{eq: mequiv} we get that $S_{i}^{(n)}\in\mathcal{G}_{L^{\infty}(\X)}$ (see also \cite[p.467-468]{Greim_Linfty}).

The function $\chi_{[0,1]}x$ can be approximated by convex combinations as follows:
\begin{equation}\label{eq: conclusion}
\begin{array}{cc}
&\left|\left|\chi_{[0,1]}x-\frac{1}{k}\sum_{n=1}^{k}\frac{1}{N_{k}}\sum_{i=1}^{N_{k}}S_{i}^{(n)}(F)\right|\right|_{L^{\infty}(\X)}
\leq \frac{1}{k}(2+\sum_{i=1}^{k-1}2k^{-1}).
\end{array}
\end{equation}
Indeed, for $n\leq k$ and a.e. $t\in [0,1]\setminus \Delta_{n}$ it holds by \eqref{eq: xsumd} that
\begin{equation*}
\begin{array}{cc}
&\left|\left|x-\frac{1}{N_{k}}\sum_{i=1}^{N_{k}}S_{i}^{(n)}(F)(t)\right|\right|_{\X}= \left|\left|x-\frac{1}{N_{k}}\sum_{i=1}^{N_{k}}T_{i}^{(n)}(x_{n})\right|\right|_{\X}\leq \frac{2}{k}.
\end{array}
\end{equation*}
On the other hand, $||x-\frac{1}{N_{k}}\sum_{i=1}^{N_{k}}S_{i}^{(n)}(F)(t)||_{\X}\leq 2$ for a.e. $t\in \Delta_{n}$.
In \eqref{eq: conclusion} we apply the fact that $\Delta_{n}$ are pairwise essentially disjoint.

Denote $\mathrm{A}_{2}=\frac{1}{k}\sum_{n=1}^{k}\frac{1}{N_{k}}\sum_{i=1}^{N_{k}}S_{i}^{(n)}\in \conv(\mathcal{G}_{L^{\infty}(\X)})$.
By combining the estimates \eqref{eq: GCC} and \eqref{eq: conclusion} we obtain by Fact \ref{fact2} that
\[||G-\mathrm{A}_{1}\mathrm{A}_{2}(F)||<\frac{5}{k}.\]
Observe that $\mathrm{A}_{1}\mathrm{A}_{2}$ is an average of $N_{k}N_{k}$ many rotations on $L^{\infty}(\X)$.
We conclude by Fact \ref{fact1} that $L^{\infty}(\X)$ is uniformly convex-transitive.

The case $1\leq p<\infty$ is a straightforward modification of the proof of \cite[Thm. 2.1]{GJK}, where one replaces
$U_{i}x_{i}$ by suitable averages belonging to $\conv(\mathcal{G}_{\X}(x_{i}))$ for each $i$.
\end{proof}

In fact it is not difficult to check the following fact: If the rotations of $L^{\infty}(\X)$ are of the Banach-Stone type, then
$L^{\infty}(\X)$ is convex-transitive if and only if each $x\in \S_{\X}$ is a uniformly big point.

We already mentioned that $\ell^{\infty}\mod c_0$ is uniformly convex-transitive as a real space.
Next we generalize this result to the vector-valued setting.
\begin{theorem}
Let $\X$ be a uniformly convex-transitive Banach space over $\K$. Then $\ell^{\infty}(\X)\mod
c_0(\X)$ (over $\K$) is uniformly convex-transitive.
\end{theorem}

\begin{proof}
Observe that the formula
\begin{equation}\label{eq: Tsum}
T((x_{n})_{n})=(S_{n}x_{\pi(n)})_{n},
\end{equation}
where $\pi\colon \N\rightarrow\N$ is a bijection and $S_{n}\in\mathcal{G}_{\X},\ n\in \N$, defines
a rotation on $\ell^{\infty}(\X)$. Also note that such an isometry $T$ restricted to $c_0(\X)$ is
a member of $\mathcal{G}_{c_0(\X)}$.

If $T\in\mathcal{G}_{\ell^{\infty}(\X)}$ is as in \eqref{eq: Tsum}, then $\widehat{T}\colon
x+c_0(\X)\mapsto T(x)+c_0(\X)$, for $x\in\ell^{\infty}(\X)$, defines a rotation
$\ell^{\infty}(\X)\mod c_0(\X)\rightarrow \ell^{\infty}(\X)\mod c_0(\X)$. Indeed, it is clear that
$\widehat{T}\colon \ell^{\infty}(\X)\mod c_0(\X)\rightarrow \ell^{\infty}(\X)\mod c_0(\X)$ is a
linear bijection. Moreover,
\[\inf_{z\in c_0(\X)}||x-z||=\inf_{z\in c_0(\X)}||T(x)-T(z)||=\inf_{z\in c_0(\X)}||T(x)-z||,\]
so that $\widehat{T}\colon \ell^{\infty}(\X)\mod c_0(\X)\rightarrow \ell^{\infty}(\X)\mod c_0(\X)$
is an isometry.

Fix $u,v\in \S_{\ell^{\infty}(\X)\mod c_0(\X)}$. If $x,y\in \ell^{\infty}(\X)$ are such that
$u=x+c_0(\X)$ and $v=y+c_0(\X)$, then
\begin{equation}\label{eq: distsup}
\dist(x,c_0(\X))=\limsup_{n\rightarrow\infty}||x_{n}||=1=\dist(y,c_0(\X))=\limsup_{n\rightarrow\infty}||y_{n}||,
\end{equation}
since $u,v\in \S_{\ell^{\infty}(\X)\mod c_0(\X)}$. Hence we may pick $x,y\in
\S_{\ell^{\infty}(\X)}$ such that $u=x+c_0(\X)$ and $v=y+c_0(\X)$.

Fix $k\in\N$, $e\in \S_{\X}$ and let $A=\{n\in \N:\ ||x_{n}||\geq 1-\frac{1}{2k}\}$. Observe that $A$ is an infinite
set by \eqref{eq: distsup}. Since $\X$ is uniformly convex-transitive, there exists $N_{(k)}\in\N$ such that
for each $n\in A$ there are $T_{1}^{(n)},\ldots,T_{N_{(k)}}^{(n)}\in \mathcal{G}_{\X}$ such that
\begin{equation}\label{eq: e}
\begin{array}{l}
\left|\left|e-\frac{1}{N_{(k)}}\sum_{l=1}^{N_{(k)}}T_{l}^{(n)}x_{n}\right|\right|<\frac{1}{k}.
\end{array}
\end{equation}

Fix $j_{(k)}\in \N$ such that
\begin{equation}\label{eq: jk}
\frac{1}{j_{(k)}}(2+(j_{(k)}-1)(\frac{1}{k}))<\frac{2}{k}.
\end{equation}
Denote by $p_{1},\ldots,p_{j_{(k)}}\in \N$ the $j_{(k)}$ first primes. Let
$\phi_{1},\ldots, \phi_{j_{(k)}}\colon \N\rightarrow \N$ be permutations such that
\begin{equation}\label{eq: phiNA}
\phi_{i}(\N\setminus A)\subset \{p_{i}^{m}|\ m\in \N\}\quad \mathrm{for}\ i\in \{1,\ldots,j_{(k)}\}.
\end{equation}
For $l\in \{1,\ldots,N_{(k)}\}$ put $S_{i,n,l}=T_{l}^{(\phi^{-1}_{i}(n))}$ if $\phi^{-1}_{i}(n)\in A$
and otherwise put $S_{i,n,l}=\I$. Define a convex combination of rotations on $\ell^{\infty}(\X)$ by letting
\begin{equation*}
\begin{array}{l}
\mathrm{A}_{1}(z)|_{n}=\frac{1}{j_{(k)}}\sum_{i=1}^{j_{(k)}}\frac{1}{N_{(k)}}\sum_{l=1}^{N_{(k)}}S_{i,n,l}(z_{\phi^{-1}_{i}(n)}),
\end{array}
\end{equation*}
where $(z_{n})_{n\in\N}\in \ell^{\infty}(\X)$. Consider $\mathrm{A}_{1}\in L(\ell^{\infty}(\X))$ and
$\overline{e}=(e,e,e,\ldots)\in \ell^{\infty}(\X)$. We obtain that
\begin{equation}\label{eq: eee}
||\overline{e}-\mathrm{A}_{1}((x_{n}))||_{\ell^{\infty}(\X)}<\frac{2}{k}.
\end{equation}
Indeed, for each $n\in \N$ it holds for at least $j_{(k)}-1$ many indices $i$ that
\begin{equation*}
\begin{array}{l}
\frac{1}{N_{(k)}}\sum_{l=1}^{N_{(k)}}S_{i,n,l}(x_{\phi^{-1}_{i}(n)})=\frac{1}{N_{(k)}}\sum_{l=1}^{N_{(k)}}T_{l}^{(\phi_{i}^{-1}(n))}(x_{\phi_{i}^{-1}(n)}),
\end{array}
\end{equation*}
where one uses the definition of $S_{i,n,l}$, \eqref{eq: phiNA} and the fact that the sets
$\{p_{i}^{m}|\ m\in \N\}, \{p_{j}^{m}|\ m\in\N\}$ are mutually disjoint for $i\neq j$.
Thus \eqref{eq: e} and \eqref{eq: jk} yield that
\begin{equation*}
\begin{array}{l}
\left|\left|e-\frac{1}{j_{(k)}}\sum_{i=1}^{j_{(k)}}\frac{1}{N_{(k)}}\sum_{l=1}^{N_{(k)}}S_{i,n,l}(x_{\phi^{-1}_{i}(n)})\right|\right|< \frac{2}{k}
\end{array}
\end{equation*}
holds for all $n\in \N$.

Next we will define another convex combination $\mathrm{A}_{2}$ of rotations on $\ell^{\infty}(\X)$ as follows.
By using again the uniform convex transitivity of $\X$ we obtain $T_{n,l}\in \mathcal{G}_{\X},\ 1\leq l\leq N_{(k)},\ n\in \N,$
such that
\begin{equation*}
\begin{array}{l}
\left|\left|y_{n}-\frac{1}{N_{(k)}}\sum_{l=1}^{N_{(k)}}T_{n,l}e\right|\right|<\frac{1}{k}
\end{array}
\end{equation*}
holds for $n\in\N$. Define
\begin{equation*}
\begin{array}{l}
\mathrm{A}_{2}(z)|_{n}=\frac{1}{N_{(k)}}\sum_{l=1}^{N_{(k)}}T_{n,l}z_{n}.
\end{array}
\end{equation*}
Combining the convex combinations yields
\[||y-\mathrm{A}_{2}\mathrm{A}_{1}x||_{\ell^{\infty}(\X)}<\frac{3}{k}\]
according to Fact \ref{fact2}. Since the applied rotations induce rotations on
$\ell^{\infty}(\X)\mod c_0(\X)$, we may consider the corresponding convex combinations in
$L(\ell^{\infty}(\X)\mod c_0(\X))$ and thus
\[||v-\widehat{\mathrm{A}_{2}\mathrm{A}_{1}}u||_{\ell^{\infty}(\X)\mod c_0(\X)}<\frac{3}{k}.\]
Tracking the formation of the convex combinations reveals that
$\widehat{\mathrm{A}_{2}\mathrm{A}_{1}}$ can be written as an average of $N_{(k)}j_{(k)}N_{(k)}$
many rotations on $\ell^{\infty}(\X)\mod c_0(\X)$.
\end{proof}

Since $C(\beta\N\setminus \N)$ is linearly isometric to $\ell^{\infty}\mod c_0$, an application
of Theorem \ref{thm: CLX} yields that $C(\beta\N\setminus\N,\X)$ is uniformly
convex-transitive if $\X$ is uniformly convex-transitive. However, let us recall that this space is linearly isometric to
$\ell^{\infty}(\X)\mod c_0(\X)$ if and only if $\X$ is finite-dimensional.

\section{Roughness and projections}

Let $\X$ be a Banach space. For each $x \in \S_{\X}$ we denote
\[\eta(\X,x)\stackrel{\cdot}{=}\limsup_{\|h\|\to 0} \frac{\|x+h\|+\|x-h\|-2}{\|h\|}.\]
Given $\varepsilon > 0$, the space $\X$ is said to be
\emph{$\varepsilon$-rough} if $\displaystyle \inf_{x\in \S_{\X}} \eta(\X,x) \geq \varepsilon$. In
addition, $2$-rough spaces are usually called \emph{extremely rough}.

We will denote the \emph{coprojection constant} of $\X$ by
\[\rho(\X)=\sup_{P}||\I-P||,\]
where the supremum is taken over all linear norm-$1$ projections $P\colon \X\rightarrow \Y$.

A Banach space $\X$ is called \emph{uniformly non-square} if there exists $a\in (0,1)$ such that
if $x,y \in \B_{\X}$ and $\|x-y\| \geq 2a$ then $\|x-y\| < 2a$. These spaces were introduced in
\cite{James} by R. C. James, who also proved that this property lies strictly between uniform
convexity and reflexivity. Next we will illustrate how the previous concepts are related.

\begin{theorem}\label{thm: rough}
Let $\X$ be a Banach space. Then the following conditions are equivalent:
\begin{enumerate}
\item[(1)]{$\X$ contains $\ell^{1}(2)$ almost isometrically.}
\item[(2)]{$\X$ is not uniformly non-square.}
\item[(3)]{$\rho(\X)=2$.}
\end{enumerate}
Moreover, if $\sup_{x\in\S_{\X}}\eta(\X,x)=2$, then $\rho(\X)=2$.
\end{theorem}

We will require some preparations before the proof. Recall that given $x,y\in \X$ the function
$t\mapsto \frac{||x+ty||-||x||}{t}$ is monotone in $t$ and thus the limit $\lim_{t\rightarrow
0^{+}}\frac{||x+ty||-||x||}{t}$ exists and is finite.

\begin{lemma}\label{lm: canonical}
Let $\X$ be a Banach space and $x,y\in \X,\ x\neq 0$. Then
\begin{equation*}
\begin{array}{lll}
& & \phantom{\Big |}\lim_{t\rightarrow 0^{+}}\frac{||x+t(y+\theta x)||-||x||}{t}=\lim_{t\rightarrow 0^{+}}\frac{||x-t(y+\theta x)||-||x||}{t}\\
&=&\phantom{\Big |}\lim_{t\rightarrow 0^{+}}\frac{||x+t(y+\theta x)||+||x-t(y+\theta x)||-2||x||}{2t}
\end{array}
\end{equation*}
for $\theta\stackrel{\cdot}{=}\lim_{t\rightarrow 0^{+}}\frac{||x-ty||-||x+ty||}{2t||x||}$.
\end{lemma}
\begin{proof}
Observe that for all maps $a\colon [0,1]\rightarrow \R$ such that $\lim_{t\rightarrow 0^{+}}a(t)>0$ it holds that
\begin{equation}
\begin{array}{lll}\label{eq: a}
& &\phantom{\Big |}\lim_{t\rightarrow 0^{+}}\frac{||a(t)x+ty||-||a(t)x||}{t}=\lim_{t\rightarrow 0^{+}}\frac{||a(t)x+\frac{a(t)}{a(t)}ty||-||a(t)x||}{t}\\
&=&\phantom{\Big |}\lim_{t\rightarrow 0^{+}}\frac{||x+\frac{t}{a(t)}y||-||x||}{\frac{t}{a(t)}}=\lim_{t\rightarrow 0^{+}}\frac{||x+ty||-||x||}{t}.
\end{array}
\end{equation}
We will also apply the fact that
\begin{equation}\label{eq: b}
\lim_{t\rightarrow 0^{+}}\frac{t(\lim_{t\rightarrow 0^{+}}
\frac{||x-ty||-||x+ty||}{2t||x||})-t\frac{||x-ty||-||x+ty||}{2t||x||}}{t}=0.
\end{equation}
The claimed one-sided limits are calculated as follows:
\begin{eqnarray*}
& &\lim_{t\rightarrow 0^{+}}\frac{||x+t(y+\theta x)||-||x||}{t}\\
&=&\lim_{t\rightarrow 0^{+}}\frac{||(1+\frac{||x-ty||-||x+ty||}{2||x||})x+ty||-||x||}{t}\\
&=&\lim_{t\rightarrow 0^{+}}\frac{||(1+\frac{||x-ty||-||x+ty||}{2||x||})x+ty||-(1+\frac{||x-ty||-||x+ty||}{2||x||})||x||}{t}\\
&+&\lim_{t\rightarrow 0^{+}}\frac{(1+\frac{||x-ty||-||x+ty||}{2||x||})||x||-||x||}{t}\\
&=&\lim_{t\rightarrow 0^{+}}\frac{||x+ty||-||x||}{t}+\lim_{t\rightarrow 0^{+}}\frac{||x-ty||-||x+ty||}{2t}\\
&=&\lim_{t\rightarrow 0^{+}}\frac{||x+ty||+||x-ty||-2||x||}{2t}.
\end{eqnarray*}
In the first equality above we applied the fact \eqref{eq: b}, and in the third equality the fact
\eqref{eq: a}. The calculations for the equation
\begin{equation*}
\lim_{t\rightarrow 0^{+}}\frac{||x-t(y+\theta x)||-||x||}{t}=\lim_{t\rightarrow 0^{+}}\frac{||x+ty||-||x-ty||-2||x||}{2t}
\end{equation*}
are similar.
\end{proof}

\begin{proof}[Proof of Theorem \ref{thm: rough}]
The equivalence of conditions (1) and (2) is well-known. The direction (1)$\implies$(3) is
established by using the Hahn-Banach Theorem to obtain suitable rank-$1$ projections $P$. 
Towards the implication (3)$\implies$(2), suppose that $\rho(\X)=2$. 
Given $\delta>0$ there exists a projection $P\colon \X\rightarrow \Y$, which satisfies $||P||=1$ and 
$||\I-P||>2-\frac{\delta}{2}$. Choose $x\in\S_{\X}$ such that $||x-P(x)||>2-\frac{\delta}{2}$. This gives that $||P(x)||\geq
1-\frac{\delta}{2}$. Put $y=\frac{P(x)}{||P(x)||}$ and note that $y\in \S_{\X}$ and
$||y-P(x)||<\frac{\delta}{2}$. Moreover,
\begin{eqnarray*}
& &||x-y||\geq ||x-P(x)||-||y-P(x)||>2-\delta>2(1-\delta)\\
\mathrm{and}& & \\
& &||x+y||\geq ||x+P(x)||-||y-P(x)||>||x+P(x)||-\frac{\delta}{2}\\
&=&||2x+P(x)-x||\ ||P||-\frac{\delta}{2}\\
&\geq &||P(2x+P(x)-x)||-\frac{\delta}{2}=||P(2x)||-\frac{\delta}{2}>2-\delta-\frac{\delta}{2}>2(1-\delta).
\end{eqnarray*}
Thus $\X$ is not uniformly non-square.

To check the latter part of the claim, an application of Lemma \ref{lm: canonical} yields that if
$\sup_{x\in \S_{\X}}\eta(x,\X)=2$, then $\X$ is not uniformly non-square. Hence $\rho(\X)=2$.
\end{proof}

The extreme roughness of $\X$ is a tremendously stronger condition than $\rho(\X)=2$. For example,
if $(F_{n})$ is a sequence of finite-dimensional smooth spaces such that $\rho(F_{n})\rightarrow
2$ as $n\rightarrow \infty$, then the space
\[\quad \quad \quad \quad \quad \quad \quad \X=\bigoplus_{n\in \N}F_{n}\quad \quad \mathrm{
(summation\ in }\ \ell^{2}\mathrm{ -sense ) }\]
is Fr\'{e}chet-smooth but $\rho(\X)=2$.

However, for convex-transitive spaces $\X$ the condition of being extremely rough is equivalent to
the condition $\rho(\X)=2$. Indeed, if a convex-transitive space is not extremely rough then,
by \cite[Thm. 6.8]{BR2}, it must be uniformly convex and thus $\rho(\X)<2$. It is unknown to us whether a convex-transitive
Banach space is reflexive if it does not contain an isomorphic copy of $\ell^1$.

In the same spirit as in this section, the projection constants of $L^p$ spaces were discussed in
\cite{asytrans}.

\section{Final Remarks: On the universality of transitivity properties}

The well-known Banach-Mazur problem mentioned in the introduction asks whether every transitive,
separable Banach space must be linearly isometric to a Hilbert space. It is well-known that all
such (transitive+separable) spaces must be smooth; otherwise, not much is known. Even adding some
properties like being a dual space or even reflexivity has not sufficed, to date, for proving that
the norm is Hilbertian.

Let us make a few remarks on the universality of some spaces of continuous functions. It is
well-known that $C(\Delta)$ contains $C([0,1])$ isometrically; hence, the former space is
universal for the property of being uniformly convex-transitive and separable. However, it is not
almost transitive.

To get a space which is universal for the property of being almost transitive and separable, just
consider the almost transitive space $X=C_0^{\C}(L)$ where $L$ is the pseudo-arc with one point
removed (\cite{Kawamura} or \cite{Rambla}). Since $[0,1]$ is a continuous image of $L$, every
separable space is isometrically contained in $X$ (complex case) or $X_{\R}$ (real case). Finally,
note that the almost transitivity of a Banach space implies that of the real underlying space.

\subsection*{Acknowledgements}
The first author was financially supported by Junta de Andaluc\'{i}a grant FQM 257 and Project MTM-2006-15546-C02-01.
The second author was financially supported by V\"{a}is\"{a}l\"{a} Foundation and Emil Aaltonen Foundation.
This paper was conceived as the second named author enjoyed the hospitality of University of C\'{a}diz in 2008.


\begin{thebibliography}{13}
\bibitem{AR}
A. Aizpuru, F. Rambla, Almost transitivity in $C_0$ spaces of Vector-valued functions,
\textit{Proc. Edinburgh Math. Soc.} {\bf 48} (2005), 513--529.
\bibitem{Ba}
S. Banach, \emph{Th\'{e}orie des Op\'{e}rations Lin\'{e}aires}, Warsaw (1932).
\bibitem{BR2}
J. Becerra, A. Rodriguez, Transitivity of the norm on Banach spaces, \textit{Extracta Math.} {\bf
17} (2002), 1--58.
\bibitem{CaTa}
Communications between F. Cabello Sanchez and J. Talponen by email in January 2009.
\bibitem{Ca?}
F. Cabello, Convex transitive norms on spaces of continuous functions, \textit{Bull. London Math.
Soc.} {\bf 37} (2005), 107--118.
\bibitem{Ca0}
F. Cabello, Regards sur le probl\`eme des rotations de Mazur (French), \textit{Extracta Math.}
{\bf 12} (1997), no. 2, 97--116.
\bibitem{DU}
J. Diestel, J. Uhl Jr, \emph{Vector measures.} Mathematical Surveys, No. 15. American Mathematical Society, Providence, R.I., 1977.
\bibitem{HHZ}
M. Fabian, P. Habala, P. Hajek, V. Montesinos Santalucia, J. Pelant, V.Zizler, \emph{Functional
Analysis and Infinite-dimensional Geometry}, CMS Books in Mathematics, Springer-Verlag 2001.
\bibitem{Greim_Linfty}
P. Greim, The centralizer of Bochner $L^{\infty}$-Spaces, \textit{Math. Ann.} {\bf 260} (1982),
463--468.
\bibitem{Greim_Lp}
P. Greim, Isometries and $L^p$-structure of separably valued Bochner $L^p$-spaces, in Measure
Theory and its Applications, Lecture Notes in Mathematics 1033, Springer-Verlag, (1983), 209--218.
\bibitem{GJK}
P. Greim, J.E. Jamison, A. Kaminska, Almost transitivity of some functions spaces,
\textit{Math. Proc. Cambridge Phil. Soc.} {\bf 116} (1994), 475-488.
\bibitem{James}
R. C. James, Uniformly non-square Banach spaces. \textit{Ann. of Math.} {\bf 80} (1964), no. 2,
542--550.
\bibitem{Kawamura}
K. Kawamura, On a conjecture of Wood, \textit{Glasg. Math. J.}, {\bf 47} (2005), 1-5.
\bibitem{Lac}
H.E. Lacey, \emph{The isometric theory of classical Banach spaces.} Die Grundlehren der mathematischen Wissenschaften, Band 208.
Springer-Verlag, New York-Heidelberg, 1974.
\bibitem{PR}
A. Pelczynski, S. Rolewicz, Best norms with respect to isometry groups in normed linear spaces, Short Communications
on International Math. Conference in Stockholm (1962), 104.
\bibitem{Rambla}
F. Rambla, A counterexample to Wood's conjecture, \textit{J. Math. Anal. Appl.} {\bf 317} (2006),
659--667.
\bibitem{Rol}
S. Rolewicz, \emph{Metric Linear Spaces}, Polish Scientific Publishers, Reidel, 1985.
\bibitem{asytrans}
J. Talponen, Asymptotically transitive Banach spaces, Banach Spaces and their Applications in
Analysis, Eds. B.Randrianatoanina, N. Randrianantoanina, de Gruyter Proceedings in Mathematics
(2007), 423--438.
\bibitem{conv}
J. Talponen, Convex-transitivity in function spaces, J. Math. Anal. Appl. (to appear).
\bibitem{Wo}
G. Wood, Maximal symmetry in Banach spaces, \textit{Proc. Roy. Irish Acad.} {\bf 82A} (1982),
177--186.
\end{thebibliography}
\end{document}